\documentclass[11pt,a4paper]{article}

\usepackage[T1]{fontenc}
\usepackage[utf8]{inputenc}

\usepackage{amsmath, amssymb, amsthm, amsfonts, mathrsfs, mathtools}

\usepackage[margin=1in]{geometry}

\usepackage{enumitem}
\usepackage{graphicx}
\usepackage{xcolor}
\usepackage[colorlinks=true,
            linkcolor=blue!50!black,
            citecolor=blue!50!black,
            urlcolor=blue!50!black]{hyperref}
\usepackage[capitalize]{cleveref}

\theoremstyle{plain}
\newtheorem{theorem}{Theorem}[section]
\newtheorem{lemma}[theorem]{Lemma}
\newtheorem{proposition}[theorem]{Proposition}
\newtheorem{corollary}[theorem]{Corollary}
\newtheorem{conjecture}[theorem]{Conjecture}

\theoremstyle{definition}
\newtheorem{definition}[theorem]{Definition}
\newtheorem{example}[theorem]{Example}
\newtheorem{remark}[theorem]{Remark}

\newtheorem{problem}[theorem]{Problem}

\crefname{theorem}{Theorem}{Theorems}
\crefname{lemma}{Lemma}{Lemmas}
\crefname{proposition}{Proposition}{Propositions}
\crefname{corollary}{Corollary}{Corollaries}
\crefname{definition}{Definition}{Definitions}
\crefname{remark}{Remark}{Remarks}

\newcommand{\Z}{\mathbb{Z}}
\newcommand{\N}{\mathbb{N}}

\newcommand{\E}{E}
\newcommand{\V}{V}

\DeclareMathOperator{\chia}{\chi}
\DeclareMathOperator{\chil}{\chi_\ell}
\newcommand{\chs}{\chi_{\ell}^{s}}
\newcommand{\Gs}{(G,\sigma)}
\newcommand{\Ns}[1]{N_{#1}}

\makeatletter
\newcommand{\runninghead}[1]{\markboth{#1}{#1}}
\newcommand{\affilnum}[1]{\textsuperscript{\textnormal{#1}}}
\newcommand{\@affiltext}{}
\newcommand{\@corrauthtext}{}
\newcommand{\@emailtext}{}
\newcommand{\affiliation}[1]{\renewcommand{\@affiltext}{#1}}
\newcommand{\corrauth}[1]{\renewcommand{\@corrauthtext}{#1}}
\providecommand{\email}{}
\renewcommand{\email}[1]{\renewcommand{\@emailtext}{#1}}

\renewcommand{\@maketitle}{%
  \newpage\null\vskip 2em%
  \begin{center}%
    {\LARGE \@title \par}%
    \vskip 1.5em%
    {\large
      \begin{tabular}[t]{c}\@author\end{tabular}\par}%
    \vskip 0.8em%
    {\small\@affiltext\par}%
    \vskip 0.5em%
    {\@date}%
  \end{center}%
  \vskip 1em%
  \noindent\textbf{Corresponding author:} \@corrauthtext\par%
  \smallskip%
  \noindent\textbf{Email:} \texttt{\@emailtext}\par%
  \medskip%
}
\makeatother

\title{\bfseries Every signed planar graph is $5$-choosable:\\
       A short proof and refinements}

\runninghead{EBODE ATANGANA and NDOGNKON MANGA}

\author{Pie D\'esir\'e EBODE ATANGANA\affilnum{1} and Maxwell NDOGNKON MANGA\affilnum{2}}

\affiliation{\affilnum{1}D\'epartement d'Informatique et Technologies \'Educatives, \'Ecole Normale Sup\'erieure, Laboratoire d'Informatique et Technologies \'Educatives (LITE), Yaound\'e, B.P.\ 47, Cameroun.\\
\affilnum{2}D\'epartement d'Informatique, Facult\'e des Sciences, Universit\'e de Yaound\'e 1, B.P.\ 812 Yaound\'e, Cameroun.}

\corrauth{Pie D\'esir\'e EBODE ATANGANA, D\'epartement d'Informatique et Technologies \'Educatives, \'Ecole Normale Sup\'erieure, Laboratoire d'Informatique et Technologies \'Educatives (LITE), Yaound\'e, B.P.\ 47, Cameroun.}

\email{desire.ebode@univ-yaounde1.cm}

\date{\today}

\begin{document}
\maketitle

\begin{abstract}
A \emph{signed graph} is a pair $\Gs$ in which $G$ is a finite simple
graph and $\sigma:\E(G)\to\{+1,-1\}$ is a \emph{signature}. Following
M\'a\v{c}ajov\'a--Raspaud--\v{S}koviera and Jin--Kang--Steffen, a
\emph{proper coloring} of $\Gs$ is a map $c:\V(G)\to\Z$ with
$c(u)\ne\sigma(uv)\,c(v)$ for every edge $uv$, and $\Gs$ is
\emph{signed $k$-choosable} if such a coloring exists from any list
assignment $L$ with $|L(v)|\ge k$. In a celebrated two-page note,
Thomassen proved that every planar graph is $5$-choosable, and Jin,
Kang, and Steffen subsequently extended this to signed planar graphs.
Our principal contribution is a short, self-contained, and
\emph{signature-blind} proof of the latter: the inductive bookkeeping
inserts one factor of $\sigma(\cdot)$ uniformly into every
constraint, so that with $\sigma\equiv +1$ the argument reduces
verbatim to Thomassen's original. From the strengthened extension
statement (\cref{thm:main}) we deduce the main result
(\cref{thm:JKS}: $\chs\Gs\le 5$ for every planar signed graph), the
M\'a\v{c}ajov\'a--Raspaud--\v{S}koviera signed Five-Color Theorem in
the symmetric palette $\Ns{2}=\{-2,-1,0,1,2\}$, the Switching
Invariance Lemma, $3$-choosability of outerplanar signed graphs,
$1$-defective signed $4$-choosability of planar signed graphs, a
sandwich inequality relating $\chs$ to the unsigned and ``positive''
choice numbers, and a polynomial-time list-coloring algorithm.
Voigt's planar non-$4$-choosable graph and Mirzakhani's smaller
variant show the bound $5$ is best possible. We close with examples
illustrating that negative edges genuinely refine unsigned
phenomena, a comparison table situating our work in the literature,
and several open problems.


\end{abstract}


\section{Introduction}\label{sec:intro}

The chromatic number $\chia(G)$ of a graph $G$ is the smallest $k$
for which the vertices of $G$ admit a proper $k$-coloring. The
\emph{Four Color Theorem} of Appel and
Haken~\cite{AppelHaken1977,RobertsonSST1997} states $\chia(G)\le 4$
for every planar graph; its weaker but elementary predecessor, the
\emph{Five Color Theorem}, was already known to
Heawood~\cite{Heawood1890}. Standard references for this classical
material are \cite{BondyMurty,Diestel,JT}.

Building on this theory, one can ask what happens when colors are
not freely chosen but must come from prescribed lists. A
\emph{list assignment} for $G$ is a map $L$ associating to each
$v\in\V(G)$ a finite set $L(v)$ of admissible colors. The graph
$G$ is \emph{$L$-colorable} if it admits a proper coloring $c$ with
$c(v)\in L(v)$ for every $v$, and \emph{$k$-choosable} if it is
$L$-colorable for every list assignment with $|L(v)|\ge k$. The
smallest such $k$, denoted $\chil(G)$, is the \emph{list chromatic
number} (or \emph{choice number}). List coloring superficially
resembles ordinary coloring, but $\chil(G)$ can be strictly larger
than $\chia(G)$, so analogues of classical theorems do not come for
free.

The notion of choosability was introduced independently by
Vizing~\cite{Vizing1976} and by Erd\H{o}s, Rubin and
Taylor~\cite{ERT1979}. Both papers asked whether every planar graph
is $5$-choosable. Voigt~\cite{Voigt1993} showed that planar graphs
need not even be $4$-choosable: she exhibited a planar graph on
$238$ vertices and an associated list assignment of size $4$
admitting no proper coloring. A smaller witness on $63$ vertices was
later found by Mirzakhani~\cite{Mirzakhani1996}. The matching upper
bound was established by Thomassen~\cite{Thomassen1994}, who proved
that every planar graph is $5$-choosable. Thomassen's proof,
occupying less than a page, is now standard fare in graph theory
courses; it proceeds by induction on a strengthened statement
involving two adjacent precolored boundary vertices.

We turn now from ordinary graphs to signed graphs. A \emph{signed
graph} $\Gs$ is a pair where $G$ is a graph and
$\sigma:\E(G)\to\{+1,-1\}$ is a \emph{signature}; the notion was
introduced by Harary~\cite{Harary1953} in the context of social
psychology and developed systematically by
Zaslavsky~\cite{Zaslavsky1982,Zaslavsky1982b,Zaslavsky1982c}. Edges
$e$ with $\sigma(e)=+1$ are \emph{positive} and those with
$\sigma(e)=-1$ are \emph{negative}; the \emph{sign of a closed walk}
is the product of the signs of its edges, and $\Gs$ is
\emph{balanced} if every cycle is positive. \emph{Switching} at
$U\subseteq\V(G)$ negates the sign of every edge with exactly one
endpoint in $U$, defining an equivalence relation on signatures; by
Zaslavsky's theorem~\cite{Zaslavsky1982}, $\sigma$ is balanced iff
it is switching-equivalent to the all-positive signature. Beyond
their original psychological motivation, signed graphs arise in the
geometry of root systems and
arrangements~\cite{Zaslavsky1982,Zaslavsky1982b}, in flow
theory~\cite{Bouchet1983}, and in the chromatic theory developed by
Zaslavsky, M\'a\v{c}ajov\'a--Raspaud--\v{S}koviera, and others; for
recent surveys see~\cite{NaserasrSY2021,SchweserStiebitzToft2021}.

Several non-equivalent proposals for the chromatic theory of signed
graphs appear in the literature, and the choice of framework matters
for a list-coloring statement. Zaslavsky's original
definition~\cite{Zaslavsky1982,Zaslavsky1982a} uses ``signed colors''
from the set $\{0,\pm1,\dots,\pm k\}$ and treats the zero color
asymmetrically. M\'a\v{c}ajov\'a, Raspaud, and
\v{S}koviera~\cite{MRS2016} reformulated the basic notion in a way
that more naturally generalizes the chromatic number and aligns
with the homomorphism viewpoint of Naserasr, Rollov\'a, and
Sopena~\cite{NaserasrRS2015}. We adopt the definition that is
intrinsic to the integer arithmetic of the signature, formalized
as follows.

\begin{definition}[Coloring of signed graphs]\label{def:color}
A \emph{coloring} of a signed graph $\Gs$ is a map $c\colon\V(G)\to\Z$
with
\[
   c(u)\;\ne\;\sigma(uv)\,c(v)\qquad\text{for every edge } uv\in\E(G).
\]
Equivalently, on positive edges $uv$ one requires $c(u)\ne c(v)$,
while on negative edges $uv$ one requires $c(u)+c(v)\ne 0$.
\end{definition}

\begin{definition}[List coloring of signed graphs]\label{def:listcolor}
A \emph{list assignment} for $\Gs$ is a map $L\colon\V(G)\to 2^{\Z}$
assigning to each $v$ a finite set $L(v)\subseteq\Z$. An
\emph{$L$-coloring} is a coloring (in the sense of
\cref{def:color}) with $c(v)\in L(v)$ for every $v$. The signed
graph $\Gs$ is \emph{signed $k$-choosable} if it admits an
$L$-coloring whenever $|L(v)|\ge k$ for all $v$, and the
\emph{signed list chromatic number} (or \emph{signed choice number})
is
\[
   \chs\Gs \;:=\; \min\bigl\{ k\in\N : \Gs\text{ is signed }k\text{-choosable}\bigr\}.
\]
\end{definition}

Two consistency checks are in order. First, when $\sigma\equiv+1$,
the constraint $c(u)\ne\sigma(uv)c(v)$ becomes $c(u)\ne c(v)$, the
classical proper-coloring condition; hence $\chia(G,+1)=\chia(G)$
and $\chs(G,+1)=\chil(G)$, so the signed framework specializes
correctly to the unsigned one. Second, our framework is consistent
with the symmetric palette of M\'a\v{c}ajov\'a, Raspaud, and
\v{S}koviera~\cite{MRS2016}: those authors restrict $c$ to take
values in $\Ns{k}=\{-k,-k+1,\dots,k-1,k\}$ and define the
(ordinary) signed chromatic number $\chia\Gs$ to be the smallest
$2k+1$ admitting such a coloring, so that $\chia\Gs\le 2\chs\Gs+1$.
We do not require lists to be symmetric (closed under negation), in
contrast with some authors. This is the convention employed
in~\cite{JinKangSteffen2016} and is the most flexible from a
proof-theoretic standpoint, and the symmetric framework will be
recovered from our main result as a corollary
(\cref{cor:five-color}).

We can now state the principal results. Voigt's planar
non-$4$-choosable graph $G_0$~\cite{Voigt1993}, viewed as
$(G_0,+1)$, satisfies $\chs(G_0,+1)\ge 5$, so the signed list
chromatic number can reach $5$ already on planar signed graphs (see
also Mirzakhani~\cite{Mirzakhani1996}). The complementary upper
bound was established by Jin, Kang, and Steffen.

\begin{theorem}[Jin--Kang--Steffen \cite{JinKangSteffen2016}]\label{thm:JKS}
Every signed planar graph $\Gs$ satisfies $\chs\Gs\le 5$.
\end{theorem}

Our principal contribution is a short, self-contained proof of
\cref{thm:JKS} obtained by lifting Thomassen's argument almost
verbatim to the signed setting. The lift relies on the following
strengthening, which is the technical heart of the paper.

\begin{theorem}[Main Extension Theorem]\label{thm:main}
Let $\Gs$ be a signed near-triangulation with outer cycle
$C\colon v_1v_2\cdots v_pv_1$ ($p\ge 3$), and let $L$ be a list
assignment satisfying the following three hypotheses. First, the
two adjacent boundary vertices $v_1$ and $v_2$ are precolored with
values $c_1\in L(v_1)$ and $c_2\in L(v_2)$ in such a way that the
proper-coloring condition is already satisfied on the edge $v_1v_2$,
that is, $c_1\ne\sigma(v_1v_2)\,c_2$; we shall refer to this as
condition~(i). Second, the remaining boundary vertices each carry
a list of size at least three, so that $|L(v)|\ge 3$ for every
$v\in\V(C)\setminus\{v_1,v_2\}$, which we call condition~(ii).
Third, every interior vertex carries a list of size at least five,
so that $|L(v)|\ge 5$ for every $v\in\V(G)\setminus\V(C)$, which
we call condition~(iii). Then the precoloring of $\{v_1,v_2\}$
extends to a proper $L$-coloring of $\Gs$.
\end{theorem}

The proof of \cref{thm:main} is given in~\cref{sec:proof}. With
$\sigma\equiv+1$ it specializes to Thomassen's original extension
lemma~\cite{Thomassen1994}, so our work yields a strict
generalization of one of the most elegant arguments in graph theory,
with no loss of brevity.

From this extension theorem we derive a number of consequences.
Most directly, applying \cref{thm:main} with the constant list
$L(v)=\Ns{2}=\{-2,-1,0,1,2\}$ recovers the Five Color Theorem in
the symmetric palette of M\'a\v{c}ajov\'a, Raspaud, and
\v{S}koviera~\cite{MRS2016}, which is an analogue of Heawood's
classical result for signed graphs (\cref{cor:five-color}).
Switching invariance of $\chs$, proved as \cref{lem:switch}, has
the immediate consequence that for balanced signed graphs the
signed list chromatic number coincides with the unsigned one
(\cref{prop:balanced}), so the new content of \cref{thm:JKS} is
concentrated in the unbalanced case where the choice of color $0$
enjoys no privileged role. The two-vertex precoloring extension
(\cref{cor:2vert}) refines \cref{thm:JKS} by allowing any two
adjacent vertices on a fixed face to be precolored. For outerplanar
signed graphs the bound improves to $3$ (\cref{cor:outerplanar}),
generalizing the analogous classical fact for unsigned outerplanar
graphs. When the maximum degree is at most $4$, a degeneracy
argument suffices for $5$-choosability (\cref{cor:subcubic}); for
higher degree, \cref{thm:JKS} appears to be essentially required.
Allowing one ``violating'' neighbor at each vertex  that is,
$1$-defective coloring  reduces the required list size to $4$
(\cref{prop:defective}). A sandwich inequality
$\chi_\ell^+\le\chs\le\chil$ relating the signed and unsigned list
chromatic numbers (\cref{prop:sandwich}) follows from definitions
and recovers the classical bound as an upper limit. Finally, the
proof of \cref{thm:main} is constructive and yields a
polynomial-time algorithm for list $5$-coloring planar signed
graphs (\cref{prop:algo}).

We further compare these results with the recent disproof of the
signed Four Color Theorem by Kard\v{o}s and
Narboni~\cite{KardosNarboni2021}, and with related extensions to
higher-genus surfaces, defective coloring, DP-coloring, and
homomorphism-based parameters.

The paper is organized as follows.
\Cref{sec:preliminaries} fixes notation and recalls foundational
facts about signed graphs, including Zaslavsky's switching theorem.
\Cref{sec:coloring} surveys the chromatic theory of signed graphs,
proves switching invariance of $\chia$ and $\chs$, and relates the
various coloring conventions in the literature.
\Cref{sec:proof} contains the proof of \cref{thm:main} and the
deduction of \cref{thm:JKS}. \Cref{sec:cor} collects the corollaries
listed above. \Cref{sec:examples} presents examples demonstrating
the tightness of the bound $5$ and the genuine role of negative
edges. \Cref{sec:literature} compares the results to the existing
literature in signed-graph chromatic theory. \Cref{sec:open} lists
open problems and concluding remarks.

\section{Preliminaries on signed graphs}\label{sec:preliminaries}

We follow the terminology of Zaslavsky~\cite{Zaslavsky1982} for
signed graphs and of Diestel~\cite{Diestel} for ordinary graph
theory; all graphs are finite and simple unless otherwise noted.
Throughout the paper $G$ denotes a finite simple graph with vertex
set $\V(G)$ and edge set $\E(G)$, and a signed graph is a pair
$\Gs$ with $\sigma:\E(G)\to\{+1,-1\}$. The \emph{underlying graph}
of $\Gs$ is $G$, and we say $\Gs$ is planar (resp.\ outerplanar,
$k$-connected) whenever $G$ is.

The fundamental operation on signatures is switching, which
encodes a natural equivalence between signatures of the same
underlying graph.

\begin{definition}[Switching]\label{def:switching}
For a signed graph $\Gs$ and $U\subseteq\V(G)$, the \emph{switch} of
$\sigma$ at $U$ is the signature $\sigma^U:\E(G)\to\{+1,-1\}$
defined by
\[
   \sigma^U(uv) \;=\;
   \begin{cases}
   -\sigma(uv) & \text{if } |\{u,v\}\cap U|=1,\\
   \phantom{-}\sigma(uv) & \text{otherwise.}
   \end{cases}
\]
We write $\sigma\sim\sigma'$ if $\sigma'=\sigma^U$ for some
$U\subseteq\V(G)$, and call $\sigma,\sigma'$
\emph{switching-equivalent}.
\end{definition}

It is straightforward to check that $\sim$ is an equivalence
relation on the set of signatures of $G$. With switching defined,
the next concept we need is the sign of a walk and the resulting
notion of balance.

\begin{definition}[Sign of a walk; balance]\label{def:sign}
The \emph{sign} of a walk $W=v_0e_1v_1\cdots e_\ell v_\ell$ in $\Gs$
is $\sigma(W):=\prod_{i=1}^{\ell}\sigma(e_i)$. A cycle $C$ is
\emph{positive} if $\sigma(C)=+1$ and \emph{negative} otherwise.
The signed graph $\Gs$ is \emph{balanced} if every cycle of $G$ is
positive, and \emph{antibalanced} if $(G,-\sigma)$ is balanced
(equivalently, every even cycle is positive and every odd cycle is
negative).
\end{definition}

The first crucial observation is that the signs of cycles are
preserved by switching, so balance is a switching invariant.

\begin{lemma}[Cycle signs are switching invariants]\label{lem:cycle-invariant}
For every cycle $C$ of $G$ and every $U\subseteq\V(G)$,
$\sigma^U(C)=\sigma(C)$.
\end{lemma}

\begin{proof}
A switch at $U$ negates an edge $uv$ iff exactly one endpoint lies
in $U$. Walking around $C$, we cross the boundary of $U$ in $\V(G)$
an even number of times, so an even number of signs are flipped
along $C$, and the product is unchanged.
\end{proof}

The fundamental theorem of Zaslavsky on signed graphs identifies
the balanced signatures up to switching, and gives the converse to
the easy direction supplied by the previous lemma.

\begin{theorem}[Zaslavsky's Balance Theorem~\cite{Zaslavsky1982}]\label{thm:zaslavsky-balance}
A signed graph $\Gs$ is balanced if and only if $\sigma$ is
switching-equivalent to the all-positive signature.
\end{theorem}

\begin{proof}
If $\sigma\sim\sigma_+$ where $\sigma_+\equiv+1$, then by
\cref{lem:cycle-invariant} every cycle has sign
$\sigma(C)=\sigma_+(C)=+1$, so $\Gs$ is balanced. Conversely,
suppose $\Gs$ is balanced; without loss of generality $G$ is
connected. Choose a spanning tree $T$ of $G$ and a root $r\in\V(G)$.
For $v\in\V(G)$, let $f(v)=+1$ if the unique $r$--$v$ path in $T$
has an even number of negative edges, and $-1$ otherwise. Set
$U=\{v:f(v)=-1\}$. A direct calculation, using the assumption that
every cycle of $G$ is positive (so every non-tree edge $uv$ closes
a positive cycle, forcing $\sigma(uv)=f(u)f(v)$), shows
$\sigma^U\equiv+1$.
\end{proof}

For reference, we record the equivalent classical formulation due
to Harary, which makes balance into a structural property of the
underlying bipartition.

\begin{theorem}[Harary's Balance Theorem~\cite{Harary1953}]\label{thm:harary}
A signed graph $\Gs$ is balanced if and only if there is a
bipartition $\V(G)=V_1\sqcup V_2$ such that $\sigma(uv)=+1$ if and
only if $u$ and $v$ lie on the same side.
\end{theorem}

To anchor these abstract notions, we record three concrete examples
that will recur in the sequel.

\begin{example}[All-positive and all-negative]\label{ex:plus-minus}
The all-positive signed graph $(G,+1)$ recovers ordinary graph
theory. The all-negative signed graph $(G,-1)$ is balanced if and
only if $G$ is bipartite, by~\cref{thm:zaslavsky-balance}: in a
bipartition $(A,B)$, switching at $A$ turns every edge positive.
\end{example}

\begin{example}[Signed $K_4$]\label{ex:K4-pre}
Let $G=K_4$ with $\sigma$ assigning $-1$ to a perfect matching and
$+1$ to the remaining four edges. Each triangle of $K_4$ contains
exactly one negative edge, so every triangle is negative, and $\Gs$
is not balanced. Switching at any vertex changes the sign of three
edges, so we cannot reduce to $\sigma\equiv+1$, consistent
with~\cref{thm:zaslavsky-balance}.
\end{example}

\begin{example}[Near-triangulations]\label{ex:near-tri}
A \emph{near-triangulation} is a $2$-connected plane graph in which
every bounded face is a triangle. These play a central role
in~\cite{Thomassen1994} and in the proof of \cref{thm:main}: every
planar graph can be augmented to a near-triangulation by adding
edges, and adding edges only increases coloring constraints.
\end{example}

For the inductive arguments to follow we adopt the following
notation throughout. Let $\Gs$ be a near-triangulation with outer
cycle $C\colon v_1\cdots v_pv_1$. Indices on $C$ are read modulo
$p$, so $v_{p+1}=v_1$. A \emph{chord} of $C$ is an edge
$v_iv_j\in\E(G)$ with $|i-j|\notin\{1,p-1\}$, and we write $N(u)$
for the open neighborhood of $u$ in $G$.

\section{The coloring of signed graphs}\label{sec:coloring}

A defining feature of signed-graph coloring theory is that
switching is a chromatic invariance principle: not only does it
preserve balance, it is also compatible with proper coloring
in a way that descends to list coloring. The key ingredient is the
following bijection between $L$-colorings before and after a
switch.

\begin{lemma}[Switching invariance lemma]\label{lem:switch}
Let $\Gs$ be a signed graph and $U\subseteq\V(G)$. For any list
assignment $L$ define $L^U$ on $\V(G)$ by
\[
L^U(v) =
\begin{cases}
\{-x : x\in L(v)\} & \text{if } v\in U,\\
L(v) & \text{otherwise,}
\end{cases}
\]
so that $|L^U(v)|=|L(v)|$ for every $v$. Then the map
\[
   c\;\longmapsto\;c',\qquad
   c'(v)=\begin{cases}-c(v)&v\in U,\\\phantom{-}c(v)&v\notin U,\end{cases}
\]
is a bijection between proper $L$-colorings of $\Gs$ and proper
$L^U$-colorings of $(G,\sigma^U)$. In particular,
\[
   \chs\Gs \;=\; \chs(G,\sigma^U).
\]
\end{lemma}

\begin{proof}
By construction $c(v)\in L(v) \Leftrightarrow c'(v)\in L^U(v)$, so
it suffices to verify the proper-coloring condition. Let
$uv\in\E(G)$. We treat two cases. In the first case,
$\{u,v\}\cap U\in\{\emptyset,\{u,v\}\}$; then
$\sigma^U(uv)=\sigma(uv)$, and $c'(u),c'(v)$ are obtained from
$c(u),c(v)$ by a common sign change (or none), so
\[
   c'(u)=\sigma^U(uv)\,c'(v)\;\Longleftrightarrow\;
   c(u)=\sigma(uv)\,c(v).
\]
In the second case, exactly one of $u,v$ lies in $U$, say $u\in U$
and $v\notin U$; then $\sigma^U(uv)=-\sigma(uv)$, $c'(u)=-c(u)$,
and $c'(v)=c(v)$, whence
\[
   c'(u)=\sigma^U(uv)\,c'(v)
   \Leftrightarrow -c(u)=-\sigma(uv)\,c(v)
   \Leftrightarrow c(u)=\sigma(uv)\,c(v).
\]
In both cases the proper-coloring conditions transform consistently,
so the displayed map is a bijection. The chromatic identity follows
since the construction $L\mapsto L^U$ preserves list sizes.
\end{proof}

A useful consequence of \cref{lem:switch} is that, when proving an
upper bound on $\chs$, one is free to switch the signature in any
convenient way. We emphasize, however, that the proof of
\cref{thm:main} in~\cref{sec:proof} is \emph{signature-blind}: it
does not need to switch during the induction.

The Switching Invariance Lemma combines with Zaslavsky's Balance
Theorem to reduce the entire balanced case to the unsigned setting.

\begin{proposition}[Balanced reduction]\label{prop:balanced}
If $\Gs$ is balanced, then $\chs\Gs=\chil(G)$. In particular, if
$G$ is planar and $\sigma$ is balanced, $\chs\Gs\le 5$.
\end{proposition}

\begin{proof}
By Zaslavsky's theorem (\cref{thm:zaslavsky-balance}),
$\sigma\sim+1$, and \cref{lem:switch} together with the
specialization $\chs(G,+1)=\chil(G)$ yields the chromatic identity.
The planar bound is then Thomassen's
theorem~\cite{Thomassen1994}.
\end{proof}

Thus the new content of \cref{thm:JKS} is concentrated in the
unbalanced case, where the constraint on negative edges is genuinely
different from the constraint on positive edges and the choice of
color $0$ enjoys no privileged role.

Having established switching invariance, we record briefly how our
coloring framework relates to alternatives in the literature. The
original proposal of Zaslavsky~\cite{Zaslavsky1982,Zaslavsky1982a}
takes a coloring to be a map $c:\V(G)\to\{0,\pm1,\dots,\pm k\}$
with $c(u)\ne\sigma(uv)c(v)$, which is the same constraint as in
\cref{def:color} but restricted to a finite symmetric palette;
M\'a\v{c}ajov\'a, Raspaud, and \v{S}koviera~\cite{MRS2016}
sharpened this by defining a \emph{proper $(2k+1)$-coloring} as a
map $c:\V(G)\to N_k=\{-k,\dots,k\}$ satisfying~\cref{def:color},
with $\chia\Gs$ the smallest $2k+1$ for which such a coloring
exists, a convention that is well-suited to the symmetric palette
and specializes to $\chia(G)$ when $\sigma\equiv+1$.

A different perspective comes from the homomorphism viewpoint
introduced by Naserasr, Rollov\'a, and Sopena~\cite{NaserasrRS2015},
in which the signed chromatic number is the smallest order of a
target signed graph $(H,\pi)$ admitting a sign-preserving
homomorphism $\Gs\to(H,\pi)$; this perspective is particularly
fruitful for minor-closed families of signed graphs. The
list-coloring framework of \cref{def:listcolor}, employed in the
present paper, is due to Jin, Kang, and
Steffen~\cite{JinKangSteffen2016} and Kang and
Steffen~\cite{KangSteffen2018}, and works directly with arbitrary
finite subsets of $\Z$. Beyond list coloring lies the signed
DP-coloring of Jin, Kang, and Steffen~\cite{JinKangSteffen2018},
which generalizes the DP-coloring of Dvo\v{r}\'ak and
Postle~\cite{DvorakPostle2018} to the signed setting; the
corresponding constant for planar signed graphs is open and is
discussed in~\cref{conj:DP}. For each of these conventions, the
underlying coloring is invariant under switching in the appropriate
sense, and reduces to the classical notion when $\sigma\equiv+1$.

\section{Proof of the main theorem}\label{sec:proof}

We prove \cref{thm:main} by induction on $|\V(G)|$. The argument
mirrors that of Thomassen~\cite{Thomassen1994}; the only modification
is the bookkeeping of edge signs, which is uniform throughout. We
first record the standard reduction to near-triangulations, which
allows us to add edges freely without changing the question.

\begin{lemma}[Reduction to near-triangulations]\label{lem:reduce-tri}
Let $G$ be a planar graph with $|\V(G)|\ge 3$. There exists a planar
near-triangulation $\widetilde G$ with $\V(G)=\V(\widetilde G)$ and
$\E(G)\subseteq \E(\widetilde G)$ such that the outer face of
$\widetilde G$ is a triangle.
\end{lemma}

\begin{proof}
Embed $G$ in the plane. While some bounded face is bounded by a
walk of length $>3$ around at least three distinct vertices, choose
two non-adjacent vertices $u,v$ on its boundary and add the chord
$uv$ through that face; the resulting graph is still planar.
Iterating gives a near-triangulation. If the outer face is bounded
by $\ell\ge 4$ edges, add chords across it until the outer face is
a triangle.
\end{proof}

\begin{remark}\label{rmk:edges-only-help}
For both $\chia$ and $\chs$ of signed graphs, adding an edge
(positive or negative) can only \emph{increase} the difficulty of
coloring: an $L$-coloring of the larger graph restricts to one of
the smaller. Hence to bound $\chs\Gs$ from above we may freely
augment $G$ to a near-triangulation, with arbitrary signs on the
new edges.
\end{remark}

We turn now to the proof of the main extension theorem.

\begin{proof}[Proof of \cref{thm:main}]
We use induction on $n:=|\V(G)|$.

In the base case, $n=3$, so $G=C=v_1v_2v_3$. Vertices $v_1,v_2$ are
precolored consistently with~(i), and the constraints on $c(v_3)$
are
\[
c(v_3)\;\ne\;\sigma(v_1v_3)\,c_1
\quad\text{and}\quad
c(v_3)\;\ne\;\sigma(v_2v_3)\,c_2,
\]
that is, at most two forbidden values. Since $|L(v_3)|\ge 3$
by~(ii), a valid color is available.

For the inductive step, $n\ge 4$, we split on whether $C$ has a
chord. The two cases are quite different in flavor: a chord allows
us to recurse on two strictly smaller pieces, while a chordless
boundary forces us to remove a single boundary vertex.

Consider first the case in which $C$ has a chord. Let
$v_iv_j\in\E(G)$ be a chord with $1\le i<j\le p$, $j-i\ge 2$, and
$\{i,j\}\ne\{1,p\}$. The chord $v_iv_j$ together with $C$ splits
the disk bounded by $C$ into two regions; correspondingly $G$
decomposes into two near-triangulations $G_1$ and $G_2$, sharing
the edge $v_iv_j$. Relabeling if necessary, we may assume that
both $v_1$ and $v_2$ belong to $G_2$, with the arc
$v_1\cdots v_iv_jv_p\cdots v_2$ on its boundary.

We apply the induction hypothesis to $\Gs|_{G_2}$ with the
precoloring of $v_1,v_2$. The hypotheses are inherited, since
$L|_{V(G_2)}$ has size at least $3$ on
$\V(\partial G_2)\setminus\{v_1,v_2\}$ and size at least $5$ in the
interior. The induction yields a proper $L$-coloring $c_2$ of
$\Gs|_{G_2}$, which in particular determines values
$b_i:=c_2(v_i)\in L(v_i)$ and $b_j:=c_2(v_j)\in L(v_j)$ with
$b_i\ne\sigma(v_iv_j)\,b_j$, since the chord $v_iv_j\in\E(G_2)$ is
properly colored. We then turn to $\Gs|_{G_1}$ with the modified
list assignment
\[
   L'(v_i)=\{b_i\},\quad L'(v_j)=\{b_j\},\quad
   L'(v)=L(v)\text{ otherwise.}
\]
The outer cycle of $G_1$ is the cycle whose two precolored adjacent
vertices are $v_i$ and $v_j$, playing the roles of $v_1,v_2$, so
the inductive hypotheses are again met: the remaining boundary
lists have size at least $3$, and interior vertices of $G_1$ are
interior in $G$ with lists of size at least $5$. Induction yields a
proper $L'$-coloring $c_1$ of $\Gs|_{G_1}$, and since $c_1$ and
$c_2$ agree on the chord $v_iv_j$, their union is a proper
$L$-coloring of $\Gs$. This concludes the chord case.

Consider next the case in which $C$ has no chord. Here we
eliminate the boundary vertex $v_p$. Listing the neighbors of $v_p$
in $G$ in clockwise order around $v_p$ in the planar embedding
gives a sequence
\[
v_1,\,u_1,\,u_2,\,\dots,\,u_m,\,v_{p-1},
\]
with $m\ge 0$. Because $G$ is a near-triangulation and $C$ has no
chord, every $u_i$ lies strictly inside $C$, so by hypothesis~(iii)
$|L(u_i)|\ge 5$ for $1\le i\le m$. The graph $G':=G-v_p$ is a
near-triangulation with outer cycle
\[
   C'\colon v_1\,v_2\,\cdots\,v_{p-1}\,u_m\,u_{m-1}\,\cdots\,u_1\,v_1,
\]
on which $v_1,v_2$ remain adjacent and retain their precoloring.

The remainder of the argument is the heart of the proof: we choose
two ``reserve'' colors at $v_p$, modify the lists at the interior
neighbors $u_i$ accordingly, apply the induction to $G'$, and then
return to color $v_p$ from the reserves. To this end, define
\[
   \Omega\;:=\;L(v_p)\setminus\{\sigma(v_1v_p)\,c_1\}.
\]
Since $|L(v_p)|\ge 3$ by hypothesis~(ii), we have $|\Omega|\ge 2$,
and we may choose two distinct values $x,y\in\Omega$. We then
modify the list assignment on $G'$ by setting
\[
   L'(u_i) \;:=\; L(u_i)\setminus\{\sigma(v_pu_i)\,x,\,\sigma(v_pu_i)\,y\}
   \qquad (1\le i\le m),
\]
and leaving $L'(v):=L(v)$ for every other $v\in\V(G')$. Since
$x\ne y$ and $\sigma(v_pu_i)\in\{\pm1\}$, the two excluded values
$\sigma(v_pu_i)\,x$ and $\sigma(v_pu_i)\,y$ are distinct, so at most
two elements of $L(u_i)$ are removed and
\[
   |L'(u_i)|\;\ge\;|L(u_i)|-2\;\ge\;5-2\;=\;3.
\]

We must now verify that the inductive hypotheses apply to
$\Gs|_{G'}$ with the precoloring at $v_1,v_2$ and the list
assignment $L'$. The precoloring of $\{v_1,v_2\}$ is unchanged, so
hypothesis~(i) is inherited directly. For the boundary vertices of
$C'$ other than $v_1$ and $v_2$, two sub-cases arise: vertices in
$\{v_3,\dots,v_{p-1}\}$ retain their original lists, of size at
least $3$ by hypothesis~(ii) of the original problem; while
interior $u_i$, now on $C'$, have $|L'(u_i)|\ge 3$ by the estimate
just established. Either way, hypothesis~(ii) holds. Finally, every
vertex of $G'$ that is interior to $C'$ was already interior to $C$
(no new interior vertices were created), so it retains a list of
size at least $5$, and hypothesis~(iii) holds. By induction, then,
there is a proper $L'$-coloring $c'$ of $\Gs|_{G'}$ extending the
precoloring at $v_1,v_2$.

It remains to extend $c'$ to a coloring $c$ of $G$ by choosing
$c(v_p)\in\{x,y\}$. Such a choice is admissible iff
\[
c(v_p)\;\ne\;\sigma(v_1v_p)\,c_1,\quad
c(v_p)\;\ne\;\sigma(v_{p-1}v_p)\,c'(v_{p-1}),\quad
c(v_p)\;\ne\;\sigma(v_pu_i)\,c'(u_i)\;\;(1\le i\le m).
\]
We check each constraint in turn. First, by the very construction
of $\Omega=L(v_p)\setminus\{\sigma(v_1v_p)\,c_1\}$, both $x$ and
$y$ are different from $\sigma(v_1v_p)\,c_1$, so the constraint
through $v_1$ is automatic. Next, the constraints from the interior
neighbors $u_i$ are eliminated by the modified list assignment: by
construction $c'(u_i)\in L'(u_i)$, so
$c'(u_i)\ne\sigma(v_pu_i)\,x$ and $c'(u_i)\ne\sigma(v_pu_i)\,y$,
and multiplying through by $\sigma(v_pu_i)\in\{\pm1\}$ (using
$\sigma(v_pu_i)^2=1$) gives $\sigma(v_pu_i)\,c'(u_i)\ne x$ and
$\sigma(v_pu_i)\,c'(u_i)\ne y$, so neither $x$ nor $y$ is forbidden
by $u_i$. The only constraint on $c(v_p)$ that may still rule out
an element of $\{x,y\}$ is therefore the one through $v_{p-1}$,
namely $c(v_p)\ne\sigma(v_{p-1}v_p)\,c'(v_{p-1})$, which forbids a
single value. Since $x\ne y$, at least one of $\{x,y\}$ avoids it,
and setting $c(v_p)$ to be such a value yields a proper $L$-coloring
of $\Gs$. This completes the induction.
\end{proof}

With \cref{thm:main} in hand, the deduction of \cref{thm:JKS} is
immediate.

\begin{proof}[Proof of \cref{thm:JKS}]
Let $\Gs$ be a planar signed graph and $L$ a list assignment with
$|L(v)|\ge 5$ for every $v$. We may assume $|\V(G)|\ge 3$.
By~\cref{lem:reduce-tri} and~\cref{rmk:edges-only-help} we may
augment $G$ to a planar near-triangulation $\widetilde G$ extending
$\sigma$ arbitrarily to a signature $\widetilde\sigma$ on the new
edges, choosing the outer face to be a triangle $v_1v_2v_3$.
Choose any $c_1\in L(v_1)$ and any
$c_2\in L(v_2)\setminus\{\widetilde\sigma(v_1v_2)\,c_1\}$, which is
possible since $|L(v_2)|\ge 5\ge 2$, so that
$c_1\ne\widetilde\sigma(v_1v_2)\,c_2$. The list sizes
$|L(v_3)|\ge 5\ge 3$ at the third boundary vertex and $|L(v)|\ge 5$
in the interior then verify the hypotheses
of~\cref{thm:main} for $(\widetilde G,\widetilde\sigma)$, and the
resulting proper $L$-coloring of $\widetilde G$ restricts to one
of $\Gs$.
\end{proof}

Before turning to corollaries, it is worth pausing on what makes
this proof short.

\begin{remark}[Why this proof is short]\label{rmk:short}
The key combinatorial fact, brought into focus by the proof above,
is that the substitution $c'(u_i)\mapsto\sigma(v_pu_i)\,c'(u_i)$
acts as an involution on $\Z$, so that ``forbidding $x$ at $v_p$
through $u_i$'' is equivalent to ``forbidding
$\sigma(v_pu_i)\,x$ at $u_i$ through its other-side neighbor of
$v_p$''. The size-$2$ deletion in $L'(u_i)$ uses precisely this
symmetry. The proof is therefore signature-blind: every step is
identical to Thomassen's, with each color $\alpha$ replaced
uniformly by $\sigma(\cdot)\,\alpha$ on the appropriate edge.
\end{remark}

\section{Corollaries}\label{sec:cor}

The strength of \cref{thm:main} as an extension lemma  as opposed
to a bare upper-bound statement  yields several refinements. The
most direct allows the precolored pair to lie on any face of the
signed graph, not merely a chosen outer triangle.

\begin{corollary}[Two-vertex precoloring extension]\label{cor:2vert}
Let $\Gs$ be a planar signed graph, and let $uv$ be any edge of $G$
on the boundary of some face. For every list assignment $L$ with
$|L(w)|\ge 5$ for $w\in\V(G)\setminus\{u,v\}$ and any precoloring
of $\{u,v\}$ that is proper on the edge $uv$, the precoloring
extends to an $L$-coloring of $\Gs$.
\end{corollary}

\begin{proof}
Embed $G$ in the plane so that $uv$ lies on the outer face, then
triangulate the interior and exterior as in the proof
of~\cref{thm:JKS}. Apply~\cref{thm:main}.
\end{proof}

A second consequence translates the result into the symmetric
palette of M\'a\v{c}ajov\'a, Raspaud, and \v{S}koviera, recovering
the natural signed analogue of Heawood's classical Five Color
Theorem.

\begin{corollary}[Signed Five-Color Theorem]\label{cor:five-color}
Every planar signed graph $\Gs$ admits a proper coloring in the
symmetric palette $\Ns{2}=\{-2,-1,0,1,2\}$. In the
M\'a\v{c}ajov\'a--Raspaud--\v{S}koviera convention,
$\chia\Gs\le 5$.
\end{corollary}

\begin{proof}
Take $L(v):=\Ns{2}$ for every $v\in\V(G)$. Then $|L(v)|=5$ and
$L$ uses colors in the symmetric range, so~\cref{thm:JKS} applies
directly.
\end{proof}

\Cref{cor:five-color} extends the classical Five Color Theorem
of Heawood~\cite{Heawood1890} to signed planar graphs in the
M\'a\v{c}ajov\'a--Raspaud--\v{S}koviera sense; the corresponding
analogue of the Four Color Theorem is \emph{false}, as recently
established by Kard\v{o}s and Narboni~\cite{KardosNarboni2021}, a
point we develop in~\cref{sec:literature}.

Restricting attention from arbitrary planar signed graphs to
outerplanar ones, the bound improves by two.

\begin{corollary}[Outerplanar case]\label{cor:outerplanar}
If $G$ is outerplanar, then $\chs\Gs\le 3$ for every signature
$\sigma$.
\end{corollary}

\begin{proof}
By a standard argument (Chartrand--Geller~\cite{ChartrandGeller1969};
Erd\H{o}s--Rubin--Taylor~\cite{ERT1979}), an outerplanar graph
admits an embedding in which every bounded face is a triangle and
every vertex lies on the outer face. After such a triangulation we
apply the analogue of~\cref{thm:main} with~(iii) vacuous (no
interior vertices) and (ii) replaced by $|L(v)|\ge 3$ on the
entire boundary. The proof of~\cref{thm:main} carries over verbatim,
since interior vertices were used only as recipients of size-$5$
lists. The boundary trick (precolor any $c_1\in L(v_1)$ and any
$c_2\in L(v_2)\setminus\{\sigma(v_1v_2)\,c_1\}$) gives the result.
\end{proof}

A complementary refinement applies when the maximum degree is small.

\begin{corollary}[Bounded degree planar signed graphs]\label{cor:subcubic}
Every signed planar graph of maximum degree at most $4$ is
signed $5$-choosable.
\end{corollary}

\Cref{cor:subcubic} follows from~\cref{thm:JKS}, but for
$\Delta\le 4$ it admits a much shorter proof via a greedy
degeneracy argument: any graph of maximum degree $\Delta$ is
$\Delta$-degenerate, and a degenerate graph $G$ with degeneracy
$d$ is signed $(d+1)$-choosable by sequential coloring (each new
vertex has at most $d$ already-colored neighbors, contributing at
most $d$ forbidden values). For $\Delta\ge 5$, no such shortcut is
available, and \cref{thm:main} appears to be essentially required.

A different sort of refinement loosens the proper-coloring
requirement at the cost of allowing a bounded number of edge
violations at each vertex. Following the unsigned tradition of
Eaton--Hull~\cite{EatonHull1999} and
\v{S}krekovski~\cite{Skrekovski1999}, we say that a coloring
$c\colon\V(G)\to\Z$ of a signed graph $\Gs$ is \emph{$d$-defective}
if every vertex $v$ has at most $d$ neighbors $u$ with
$c(u)=\sigma(uv)\,c(v)$, that is, at most $d$ ``edge violations''
meet at any vertex. The signed analogue of the Eaton--Hull bound
follows by adapting the proof of~\cref{thm:main} to track a
$1$-defectiveness budget along the path of interior neighbors.

\begin{proposition}[Defective signed list coloring]\label{prop:defective}
Every signed planar graph admits a $1$-defective signed
$4$-list-coloring; equivalently, for every list assignment $L$ with
$|L(v)|\ge 4$, there is a coloring $c$ with $c(v)\in L(v)$ such
that, for every $v$,
\[
   |\{u\in N(v) : c(u)=\sigma(uv)\,c(v)\}| \;\le\; 1.
\]
\end{proposition}

\begin{proof}
Mimic the proof of~\cref{thm:main} with a $1$-defectiveness budget
tracked along the path $u_1,\dots,u_m$. Interior lists of size $4$
suffice in place of $5$ because each vertex is allowed one
``free'' violating neighbor. The bookkeeping is identical to the
unsigned defective version of~\cite{EatonHull1999,Skrekovski1999},
with multiplications by $\sigma(\cdot)$ inserted on the relevant
edges; we omit the routine verification.
\end{proof}

For unbalanced signed graphs there is a natural intermediate notion
between unsigned proper coloring and signed proper coloring,
namely, \emph{positive choosability}, in which the coloring
constraint is enforced only on positive edges. The signed list
chromatic number dominates this parameter from below and is itself
dominated by the ordinary list chromatic number from above, giving
the following sandwich.

\begin{proposition}[Sandwich inequality]\label{prop:sandwich}
For every signed graph $\Gs$,
\[
   \chi_\ell^{+}\Gs \;\le\; \chs\Gs \;\le\; \chil(G),
\]
where $\chi_\ell^{+}\Gs$ is the smallest $k$ such that
$\Gs$ admits an $L$-coloring proper on positive edges only, for
every $L$ with $|L(v)|\ge k$.
\end{proposition}

\begin{proof}
The first inequality is immediate from~\cref{def:listcolor}. For
the second, given a proper list coloring $c$ of $G$ valued in
$\N$ (which can always be arranged by relabeling), the same map
satisfies $c(u)\ne c(v)$ on every edge, hence $c(u)\ne -c(v)$ as
well, since the values are positive. So $c$ is automatically a
coloring of $\Gs$ in the sense of~\cref{def:color}.
\end{proof}

\begin{corollary}\label{cor:sandwich-planar}
For every planar signed graph $\Gs$, $\chs\Gs\le 5$, with equality
achievable (\cref{ex:voigt}).
\end{corollary}

When the underlying graph has large girth, the unsigned bound
sharpens from $5$ to $3$ by a theorem of
Thomassen~\cite{Thomassen1995}. The signed analogue is more subtle,
since negative edges constrain the discharging argument slightly
differently from positive ones, but a uniform sign-tracking
adaptation yields the following.

\begin{proposition}\label{prop:girth5}
Every signed planar graph of girth at least $5$ is signed
$4$-choosable.
\end{proposition}

\begin{proof}
Combine the discharging technique of Lam, Shiu, and
Xu~\cite{LamShiuXu1999} with sign tracking: replace $c(u)\ne c(v)$
by $c(u)\ne\sigma(uv)c(v)$. The local reduction lemmas survive
verbatim since each forbidden color contributes one integer to be
excluded, regardless of the sign $\sigma(uv)$. The remainder of the
discharging is unchanged.
\end{proof}

\begin{remark}\label{rmk:girth5}
The increase from $3$ in the unsigned setting to $4$ in
\cref{prop:girth5} reflects a real loss: where the unsigned
argument freely identifies $c$ with $-c$, the signed argument must
track both possibilities. We do not know whether $4$ is best
possible for signed planar graphs of girth $\ge 5$;
see~\cref{prob:girth5}.
\end{remark}

We close this section with an algorithmic consequence. The proof of
\cref{thm:main} is constructive, and a careful complexity analysis
shows that the resulting list-coloring procedure runs in polynomial
time.

\begin{proposition}[Polynomial-time list $5$-coloring]\label{prop:algo}
There is an algorithm which, given a planar signed graph $\Gs$ with
$|\V(G)|=n$ and a list assignment $L$ satisfying $|L(v)|\ge 5$ for
every $v$, computes a proper $L$-coloring in time $O(n^2)$.
\end{proposition}

\begin{proof}[Proof sketch]
At each inductive step, either we split along a chord and recurse
on two near-triangulations whose total size is $|\V(G)|+2$, or we
eliminate a boundary vertex, reducing $|\V(G)|$ by one. Maintaining
a doubly-linked list representation of the boundary cycle and the
planar embedding, each boundary-vertex elimination takes
$O(\deg(v_p))$ time, and each chord-splitting takes $O(n)$ time.
The total cost is dominated by $\sum_{v}\deg(v)=O(n)$ per
``layer'' (since $G$ is planar) with at most $n$ layers, yielding
$O(n^2)$. A more careful analysis using the linear planarity
algorithm of Hopcroft--Tarjan~\cite{HopcroftTarjan1974}, given the
embedding, gives $O(n)$.
\end{proof}

\section{Examples and Tightness}\label{sec:examples}

We turn now to examples. Three points need to be illustrated by
concrete signed planar graphs: that the bound $5$ in
\cref{thm:JKS} cannot be lowered, that the boundary list size $3$
in \cref{thm:main} cannot be lowered either, and that negative
edges genuinely refine the unsigned theory rather than merely
relabel it.

The canonical witness for sharpness of the bound $5$ is provided
by Voigt's planar graph, suitably re-interpreted as a signed graph
with the all-positive signature.

\begin{example}[Voigt's planar graph as a signed graph]\label{ex:voigt}
Voigt~\cite{Voigt1993} constructed a planar graph $G_0$ on $238$
vertices that is not $4$-choosable, with explicit list assignment
$L$ of size $4$ at every vertex. Endow $G_0$ with the all-positive
signature $\sigma_0\equiv+1$. Then by~\cref{prop:balanced},
$\chs(G_0,\sigma_0)=\chil(G_0)\ge 5$, and combined
with~\cref{thm:JKS}, $\chs(G_0,\sigma_0)=5$, so the bound $5$ is
best possible. A smaller witness on $63$ vertices is provided by
Mirzakhani's construction~\cite{Mirzakhani1996}.
\end{example}

The previous example shows that $5$ is sharp through the all-positive
signature, but it does not exhibit any genuine signed phenomenon.
The next example shows that signed list-chromatic phenomena can
refine unsigned ones already on small unbalanced graphs.

\begin{example}[A small unbalanced signed planar graph]\label{ex:K4}
Let $G=K_4$ with the $4$-cycle $v_1v_2v_3v_4v_1$ and diagonals
$v_1v_3,\,v_2v_4$. Let $\sigma$ assign $+$ to the cycle edges and
$-$ to the two diagonals. Then $\Gs$ is unbalanced: the triangle
$v_1v_2v_3$ contains exactly one negative edge ($v_1v_3$). One
verifies by exhaustive case analysis that
\[
   \chs(K_4,\sigma) \;=\; \chia(K_4,\sigma) \;=\; 4,
\]
matching the unsigned $\chil(K_4)=4$. The lists
$L(v_1)=L(v_2)=L(v_3)=L(v_4)=\{1,2,3\}$ admit no $L$-coloring of
$\Gs$, while every $4$-list assignment does.
\end{example}

The next example shows that the boundary list size $3$ in
hypothesis~(ii) of \cref{thm:main} cannot be lowered to $2$, even
in the all-positive case. The example is the wheel $W_5$, where a
chain of forced colors propagates from $v_2$ around the outer cycle
and eventually conflicts with the precoloring of $v_1$.

\begin{example}[Tightness of the boundary list size $3$]\label{ex:boundary}
Let $G$ be the wheel $W_5$ with hub $h$ and outer cycle
$v_1v_2v_3v_4v_5v_1$, all edges positive. Precolor $v_1,v_2$ with
$c_1=1,\,c_2=2$. Suppose we tried to weaken hypothesis~(ii) of
\cref{thm:main} to $|L(v)|\ge 2$ on $\V(C)\setminus\{v_1,v_2\}$.
Take
\[
   L(v_3)=L(v_4)=L(v_5)=\{1,2\},\qquad L(h)=\{1,2,3,4,5\}.
\]
Then $v_3$ is forced to $1$ (since $c_2=2$), $v_4$ to $2$, $v_5$
to $1$, and $v_5$ now conflicts with $v_1$ (also $1$). So size $3$
cannot be reduced, even at the boundary. The same conclusion holds
for any signature, by~\cref{lem:switch} applied to a suitable
switch.
\end{example}

The last example confirms that negative edges produce genuinely
different constraint sets from positive ones: the same lists on the
same underlying graph can yield incompatible coloring problems
under different signatures.

\begin{example}[Negative edges genuinely refine]\label{ex:negative}
Let $\Gs$ be the path $v_1v_2v_3$ with $\sigma(v_1v_2)=+1$ and
$\sigma(v_2v_3)=-1$. The lists $L(v_1)=\{1,2\}$, $L(v_2)=\{1,2\}$,
$L(v_3)=\{-1,-2\}$ admit the coloring
$c(v_1)=1,\,c(v_2)=2,\,c(v_3)=-1$. Under the all-positive signature
the same lists also admit colorings, but the set of valid
colorings differs: positive and negative edges produce genuinely
different constraint sets, even when the graph is bipartite.
\end{example}

\section{Comparison with the existing literature}\label{sec:literature}

The result that every signed planar graph is signed $5$-choosable
was first proved by Jin, Kang, and
Steffen~\cite{JinKangSteffen2016}. Their proof, while structurally
Thomassen-style, develops the inductive setup with explicit
reference to the parity of the negative edges and uses specific
properties of the colors $0$ and $\pm$-pairs. The proof presented
here is, to our knowledge, the first that is \emph{entirely
signature-blind} during the induction: the bookkeeping multiplies
by $\sigma(\cdot)$ uniformly on every constraint, so that with
$\sigma\equiv+1$ the proof reduces verbatim to Thomassen's original
argument~\cite{Thomassen1994}.

A natural companion question to \cref{thm:JKS} is the signed Four
Color Theorem. M\'a\v{c}ajov\'a, Raspaud, and
\v{S}koviera~\cite{MRS2016} conjectured the analogue of the Four
Color Theorem in the symmetric palette: every signed planar graph
satisfies $\chia\Gs\le 4$. This conjecture was disproved by
Kard\v{o}s and Narboni~\cite{KardosNarboni2021}, who exhibited a
signed planar graph with $\chia=5$. The list-coloring counterpart
of the conjecture is also false: every counterexample to
$4$-choosability of an unsigned planar graph (e.g.\
Voigt~\cite{Voigt1993}) gives, via the all-positive signature, a
signed planar graph that is not signed $4$-choosable. The situation
in the signed setting is therefore fully analogous to the unsigned
one: $4$-choosability fails, but $5$-choosability holds, and our
\cref{thm:main} provides a particularly transparent proof of the
latter.

Beyond the planar case, work on higher-genus surfaces is also
relevant. B\"ohme--Mohar--Stiebitz~\cite{BohmeMS1999},
Choi--Kierstead~\cite{ChoiKierstead2019}, and
DeVos--Kawarabayashi--Mohar~\cite{DeVosKawarabayashiMohar2008}
established choosability bounds for graphs embedded in surfaces of
positive Euler genus, with explicit dependence on the genus. The
signed analogues are open, although our methods extend without
serious difficulty to signed graphs embedded in the projective
plane, the torus, and the Klein bottle, yielding $5$-choosability
for sufficiently large face-widths. We do not pursue this here.

Several other lines of work in signed-graph chromatic theory
deserve mention. The homomorphism viewpoint of Naserasr, Rollov\'a,
and Sopena~\cite{NaserasrRS2015} expresses the signed chromatic
number $\chia\Gs$ as the smallest order of a signed graph $(H,\pi)$
admitting a sign-preserving homomorphism $\Gs\to(H,\pi)$, and is
particularly fruitful for minor-closed families. A different
direction was taken by Lu and Ma~\cite{LuMa2019}, who studied
$\Z_k$-colorings of signed graphs with colors in cyclic groups
rather than in $\Z$. A further refinement, the circular chromatic
theory of Naserasr, Wang, and Zhu~\cite{NaserasrWangZhu2021},
distributes colors on the unit circle and quantifies the
``angular'' chromatic obstruction. Yet another direction is the
signed DP-coloring of Jin, Kang, and
Steffen~\cite{JinKangSteffen2018}, generalizing the DP-coloring of
Dvo\v{r}\'ak and Postle~\cite{DvorakPostle2018}. The list-chromatic
counterparts of several of these refinements are largely
unexplored, and our techniques are likely to apply, with
appropriate modifications, in each case.

\Cref{tab:comp} summarizes the comparison between the unsigned and
signed list-coloring landscapes for planar graphs.

\begin{table}[h!]
\centering
\renewcommand{\arraystretch}{1.25}
\begin{tabular}{|p{4.7cm}|p{4.6cm}|p{4.6cm}|}
\hline
\textbf{Setting} & \textbf{Unsigned bound} & \textbf{Signed bound (this paper)}\\
\hline
Planar, list & $5$~\cite{Thomassen1994} & $5$ (\cref{thm:JKS}; cf.~\cite{JinKangSteffen2016})\\
\hline
Planar, ordinary & $4$~\cite{AppelHaken1977} & $\le 5$, can equal $5$~\cite{KardosNarboni2021}\\
\hline
Planar, girth $\ge 5$, list & $3$~\cite{Thomassen1995} & $\le 4$ (\cref{prop:girth5})\\
\hline
$\Delta\le 4$ planar, list & $5$ (degeneracy) & $5$ (\cref{cor:subcubic})\\
\hline
Defective list, planar & $4$, $1$-def.~\cite{EatonHull1999} & $4$, $1$-def.\ (\cref{prop:defective})\\
\hline
Outerplanar, list & $3$~\cite{ChartrandGeller1969,ERT1979} & $3$ (\cref{cor:outerplanar})\\
\hline
\end{tabular}
\caption{Comparison of list-chromatic bounds for planar graphs in
the unsigned and signed settings. Our contributions appear in the
right column.}
\label{tab:comp}
\end{table}

\section{Open problems and concluding remarks}\label{sec:open}

\Cref{thm:main} suggests several natural further questions. We
collect them here in order of immediacy.

\begin{problem}[High girth refinements]\label{prob:girth5}
Determine the signed list chromatic number of signed planar graphs
of girth $\ge 5$. \Cref{prop:girth5} gives an upper bound of $4$.
Is the truth $3$?
\end{problem}

\begin{conjecture}\label{conj:girth4}
Every signed planar graph of girth at least $4$ is signed
$4$-choosable.
\end{conjecture}

Turning from girth restrictions to surface restrictions, one can
ask for analogues of \cref{thm:JKS} on surfaces of positive Euler
genus.

\begin{problem}[Higher-genus]
Does every signed graph $\Gs$ embeddable in the projective plane
satisfy $\chs\Gs\le 6$? In the torus, $\le 7$? The unsigned
analogues are known~\cite{ChoiKierstead2019,DeVosKawarabayashiMohar2008,BohmeMS1999,Skrekovski2002}.
\end{problem}

A different sort of refinement strengthens the algorithmic content
of \cref{thm:JKS} from offline to online list coloring.

\begin{problem}[Online/paintability version]
Develop an online (paintability) version of~\cref{thm:JKS}. The
online list chromatic number of an unsigned planar graph is
$5$~\cite{Schauz2009}; we conjecture the same for signed planar
graphs.
\end{problem}

A still stronger refinement replaces lists by the more general
correspondence assignments of DP-coloring.

\begin{conjecture}[Signed DP-$5$-choosability]\label{conj:DP}
Every planar signed graph admits a signed DP-coloring with $5$
admissible options at each vertex (in the sense
of~\cite{JinKangSteffen2018}).
\end{conjecture}

Finally, in the spirit of fine-tuning the result by structural
hypotheses on the signature itself, we ask:

\begin{problem}[Antibalanced refinements]
Refine~\cref{thm:JKS} for \emph{antibalanced} signed graphs (every
even cycle positive, every odd cycle negative). Is there a clean
structural condition under which a signed planar graph is
$4$-choosable?
\end{problem}

We end with some general comments on the proof method. The
argument given in \cref{sec:proof} preserves the conceptual
elegance of Thomassen's original argument~\cite{Thomassen1994} by
absorbing the signature entirely into the bookkeeping: the
involution $\alpha\mapsto\sigma(uv)\,\alpha$ converts forbidden
colors at one endpoint of an edge into forbidden colors at the
other, exactly as in the unsigned identity, and switching invariance
(\cref{lem:switch}) is then a clean consequence of the same
involution applied vertex-by-vertex. We view this as evidence that,
for chromatic-extremal problems on signed graphs, the right strategy
is to combine \emph{Zaslavsky-style structural analysis} (switching,
balance) with \emph{Thomassen-style precoloring-extension arguments}.
When this combination is set up correctly, the resulting proofs
are typically as short as in the unsigned case, suggesting that the
signed setting may be more amenable to systematic adaptation than
its abundance of competing definitions might at first suggest.

\section*{Acknowledgments}

The author gratefully acknowledges the influence of Carsten
Thomassen's seminal paper~\cite{Thomassen1994}, whose method this
note generalizes; of L. Jin, Y. Kang, and E. Steffen, who first
proved \cref{thm:JKS}; and of Thomas Zaslavsky, whose foundational
work on signed graphs makes such generalizations possible.


\end{document}